\documentclass{amsart}
\usepackage{amssymb,amsmath,latexsym}
\usepackage{amsthm}
\usepackage{fontenc}
\usepackage{amssymb}
\DeclareMathOperator{\Res}{Res}

\numberwithin{equation}{section}

\newtheorem{theorem}{Theorem}[section]

\begin{document}
\author{Alexander E Patkowski}
\title{On some mellin transforms for the Riemann zeta function in the critical strip}

\maketitle
\begin{abstract}We offer two new Mellin transform evaluations for the Riemann zeta function in the region $0<\Re(s)<1.$ Some discussion is offered in the way of evaluating some further Fourier integrals involving the Riemann xi function.\end{abstract}

\keywords{\it Keywords: \rm Mellin transform; Riemann zeta function; Digamma function.}

\subjclass{ \it 2010 Mathematics Subject Classification 11M06, 33C15.}

\section{Introduction and Main results} The Riemann zeta function, given by the series $$\sum_{n\ge1}n^{-s},$$ and convergent when $\Re(s)>1,$ is of great importance in the theory of numbers. Particularly important is its properties in the critical strip $0<\Re(s)<1,$ as the Riemann hypothesis says that all the nontrivial zeros are within this strip. Many integral evaluations in this strip are known, some of which have shed some light on nontrivial zeros of $\zeta(s)$ [8, 11]. One integral relevant to our study is sometimes attributed to Kloosterman [5], [11, pg.34]
\begin{equation} \int_{0}^{\infty}t^{s-1}(\psi^{(0)}(t+1)-\log(t))dt=-\frac{\pi\zeta(1-s)}{\sin(\pi s)},   \end{equation}
valid when $0<\Re(s)<1.$ Here $\psi^{(n)}(x)=\frac{\partial^{n+1} }{\partial x^{n+1}}(\log(\Gamma(x)),$ where $\Gamma(x)$ is the classical gamma function [1]. This classical result has also been used by Whittaker and Watson (see [11, pg.34]) to investigate properties of $\log\Gamma(x)$ and appears in the recent work of Dixit et al. [2, 5], which we shall relate to in the following section. The proof of this result in Titshmarch [11, pg. 29] involves application of the M$\ddot{u}$ntz formula and integration by parts. We adapted an alternative proof of (1.1) using the calculus of residues to obtain two integral formulae that appear to be new. \par Note the Stieltjes constants are given by [1]
$$\gamma_n=\lim_{r\rightarrow\infty}\left(\sum_{k\ge1}^{r}\frac{(\log(k))^n}{k}-\frac{(\log(r))^{n+1}}{n+1}\right).$$

\begin{theorem} Suppose that $0<\Re(s)<1.$ Define the function $\Lambda_1(x)$ for $x>0$ to be
$$\Lambda_1(x):=x\sum_{n\ge1}\frac{\log(x/n)}{n(x-n)}-\frac{1}{2}\left(\log^2(x)-2\gamma\log(x)-2\gamma_1+\frac{\pi^2}{3}\right),$$
then \begin{equation}\int_{0}^{\infty}t^{s-1}\Lambda_1(t)dt=\frac{\pi^2\zeta(1-s)}{\sin^2(\pi s)}.\end{equation}

Further, define the function $\Lambda_2(x)$ for $x>0$ to be
$$\Lambda_2(x):=x\sum_{n\ge1}\frac{\pi^2+(\log(x/n))^2}{n(x+n)}$$
$$-[-\gamma_2+2\gamma_1\log(x)-\frac{\log^3(x)}{3}-\gamma(\pi^2+\log^2(x))-\pi^2\log(x)],$$
then \begin{equation}\int_{0}^{\infty}t^{s-1}\Lambda_2(t)dt=\frac{2\pi^3\zeta(1-s)}{\sin^3(\pi s)}.\end{equation}

\end{theorem}

\begin{proof} We work with two known Mellin transforms [6] valid in the strip $0<\Re(s)<1$

\begin{equation} \int_{0}^{\infty}t^{s-1} \frac{\log(t)}{t-1}dt=\frac{\pi^2}{\sin^2(\pi s)},\end{equation}
\begin{equation} \int_{0}^{\infty}t^{s-1} \frac{\pi^2+\log(t)^2}{t+1}dt=\frac{2\pi^3}{\sin^3(\pi s)}.\end{equation}

For (1.2), we first note that for $-1<\Re(s)=c<0$
\begin{equation} x\sum_{n\ge1}\frac{\log(x/n)}{n(x-n)}=\frac{1}{2\pi i}\int_{(c)}\frac{\pi^2\zeta(1-s)}{\sin^2(\pi s)}x^{-s}ds\end{equation}
We now replace $s$ with $1-s$ and note that $\zeta(s)\Gamma^2(s)\Gamma^2(1-s)$ has a pole of order three at $s=1,$ and thereby move the line of integral from the region $1<\Re(s)<2$ to $0<\Re(s)<1.$ We compute that
$$\Res|_{s=1}(\pi^2\csc^2(\pi s)\zeta(s)x^s)=\frac{x}{2}[(\log(x)^2-2\gamma\log(x)-2\gamma_1+\frac{\pi^2}{3}],$$
where we applied the known formula $\Gamma'(1)=-\gamma.$ After this computation, we again replace $s$ by $1-s$ in the contour integral to find that we have the inverse relation of (1.2). \par In the case of (1.3), we use the same approach but compute the pole of order four at $s=1$ of $\zeta(s)\Gamma^3(s)\Gamma^3(1-s).$ This is 
$$\Res|_{s=1}(\pi^3\csc^3(\pi s)\zeta(s)x^s)=\frac{x}{2}[-\gamma_2+2\gamma_1\log(x)-\frac{\log^3(x)}{3}-\gamma(\pi^2+\log^2(x))-\pi^2\log(x)].$$ 
\end{proof}

If we take into consideration the double poles at strictly negative integers $s=-n<0,$ $n\in\mathbb{N},$ we find that moving the integral on the right side of (1.6) to the left gives the interesting series expansion for $|x|<1$
$$x\sum_{n\ge1}\frac{\log(x/n)}{n(x-n)}=-\sum_{n\ge1}\left(\log(x)\zeta(1+n)+\zeta'(1+n)\right)x^n,$$
and similarly,
$$x\sum_{n\ge1}\frac{\pi^2+(\log(x/n))^2}{n(x+n)}=\sum_{n\ge1}\left(\zeta''(1+n)+2\log(x)\zeta'(1+n)+\pi^2\zeta(1+n)+\zeta(1+n)\log^2(x)\right)(-x)^n.$$ Standard manipulations also show that $\Lambda_1(x)$ has the form as the integral
$$\int_{0}^{\infty}\frac{\psi(t+1)-\log(t)}{x+t}dt.$$

\section{Applications to other evaluations}
We shall offer some applications to evaluating Riemann xi function integrals which have been studied by many other others [2, 3, 4, 5, 7, 9, 10]. As usual, we write $\Xi(t):=\xi(\frac{1}{2}+it),$ where [8] $\xi(s):=\frac{1}{2}s(s-1)\pi^{-\frac{s}{2}}\Gamma(\frac{s}{2})\zeta(s).$ Hardy and Koshlyakov [7, 9] give 
\begin{equation} \int_{0}^{\infty}\frac{\Xi(t)}{\frac{1}{4}+t^2}\frac{\cos(xt)}{\cosh(\pi t)}dt=e^{x/2}\int_{0}^{\infty}(\psi(t+1)-\log(t))e^{-\pi t^2e^{2x}}dt,\end{equation}
by using a method of converting a Fourier cosine transform in to a Mellin transform as outlined in [11]. We can apply the same approach to our integrals, since they are similar in nature to (1.1), which can be used to prove (2.1). 
\begin{theorem} We have,
\begin{equation} \int_{0}^{\infty}\frac{\Xi(t)}{\frac{1}{4}+t^2}\frac{\cos(xt)}{\cosh^2(\pi t)}dt=e^{x/2}\int_{0}^{\infty}\Lambda_1(t)e^{-\pi t^2e^{2x}}dt,\end{equation}
and
\begin{equation} 2\int_{0}^{\infty}\frac{\Xi(t)}{\frac{1}{4}+t^2}\frac{\cos(xt)}{\cosh^3(\pi t)}dt=e^{x/2}\int_{0}^{\infty}\Lambda_2(t)e^{-\pi t^2e^{2x}}dt.\end{equation}

\end{theorem}
\begin{proof} In (2.2) we apply (1.2) with Parseval's theorem for Mellin transforms and the function $e^{-\pi (yt)^2}.$ Similarly in (2.3) we apply (1.3) instead. The remaining details are left for the reader. \end{proof}

1390 Bumps River Rd. \\*
Centerville, MA
02632 \\*
USA \\*
E-mail: alexpatk@hotmail.com
\end{document}